\documentclass[12pt]{article}

\usepackage{hyperref}

\usepackage{amsfonts, amsmath, amssymb, amsthm}
\usepackage{a4}

\usepackage{blkarray}
\usepackage{bbm}
\usepackage{upgreek}

\usepackage{bm}

\usepackage{tikz}
\usepackage{subcaption}
\usepackage{graphicx}

 \usepackage{xcolor}

 \definecolor{Refkey}{RGB}{255,127,0}
 \definecolor{Labelkey}{RGB}{127,0,255}
 \makeatletter 
  \def\SK@refcolor{\color{Refkey}}
  \def\SK@labelcolor{\color{Labelkey}}
 \makeatother

  \definecolor{mdg}{RGB}{0,177,0} 
  \definecolor{mdb}{RGB}{0,0,191}
  \definecolor{mddb}{RGB}{0,0,91}
  \definecolor{mdy}{RGB}{255,69,0} 
  \definecolor{gray}{RGB}{99,99,99}

\newtheorem{proposition}{Proposition}
\newtheorem{theorem}[proposition]{Theorem}
\newtheorem{lemma}{Lemma}

\theoremstyle{definition}

\newtheorem{ir}{Important Remark}

\theoremstyle{remark}
\newtheorem*{remark}{Remark}

\allowdisplaybreaks[3]

\title{Combinatorial correlation functions in three-dimen\-sional eight-vertex models}
\author{Igor G. Korepanov}

\date{November 2023 -- February 2024}

\begin{document}

 \sloppy

\maketitle

\medskip

\begin{abstract}
A new version of the self-similarity spin transform on three-dimen\-sional cubic lattices is proposed that makes possible calculation of nontrivial spin correlations in a ``combinatorial'' model, in which all permitted spin configurations have equal probabilities.\end{abstract}

\section{Introduction}\label{s:i}

\subsection{What this paper is about}\label{ss:a}

In this paper, we deal with finite parts of a three-dimen\-sional \emph{cubic lattice}~$\mathbb Z^3$. This lattice consists of points in Euclidean space~$\mathbb R^3$ having integer coordinates; below we call these points \emph{vertices}. Besides vertices, we will need \emph{edges}---these are, by definition, straight line segments of length one joining two vertices. Each edge is, this way, parallel to one of coordinate axes. Moreover, each edge is \emph{directed}, namely in the (positive) direction of the corresponding coordinate axis.

When we speak about a finite part~$\mathcal P$ of our lattice, we always assume that it includes, together with any of its vertices~$v$, all six of its adjacent edges. If such an edge~$e$ has no vertex on its other end (that is, that vertex is not in~$\mathcal P$), then we say that $e$ is an \emph{input} edge for~$\mathcal P$ if $e$ is directed towards~$v$, or \emph{output} edge if $e$ is out of~$v$.

Next, we assume that there is a ``spin'' variable attached to each edge of~$\mathcal P$ and taking values in a set of two elements. Below, we identify that set with the smallest finite field~$\mathbb F_2$; its elements are denoted simply as $0$ and~$1$, if no confusion with decimal digits is expected.

There are, this way, six spin variables around each vertex; if their values are given, we say that a \emph{local spin configuration} is given. Suppose we have declared some of such configurations \emph{permitted} (and the rest of them \emph{prohibited}). If there are spin values given on all edges of~$\mathcal P$ and such that their configuration around each vertex in~$\mathcal P$ is permitted, we say that there is a permitted spin configuration on the whole~$\mathcal P$.

Assume now that all permitted spin configurations on~$\mathcal P$ have equal probabilities (and prohibited ones have probability zero; we are thus in a ``classical definition of probability'' context). We can then pose the question of calculating different \emph{correlation functions}---we use this term broadly for any ``joint'' probabilities and, more specifically, we consider in this paper the probabilities of the events that two, three, or four different spins all take the value zero.

This problem is hard and not expected to be solved in full generality. We begin here with a relatively simple case when three output spins, for each separate vertex, depend \emph{$\mathbb F_2$-linearly} on three input spins; dependence is given by a $3\times 3$ matrix~$A$ with entries in~$\mathbb F_2$, the same for all vertices. This implies that there are only eight permitted among $2^6=64$ local spin configurations, and that is why we call our construction ``eight-vertex model''. This model may look a bit rigid; nevertheless, it is genuinely three-dimen\-sional, and reveals both fascinating mathematical structures and nontrivial calculation results for the correlations.

\subsection{The methods used}\label{ss:m}

We further develop the ``algebraic self-similarity'' theory for models associated with finite fields. Here we do it for the case of field characteristic two.

Recall that algebraic self-similarity was discovered in~\cite{block-spin}. It was inspired, from the physics side, by the Kadanoff--Wilson theory~\cite{Kadanoff,Wilson}, and from the side of mathematics, by Hietarinta's work~\cite{Hietarinta}. An important difference from Hietarinta's paper is, however, that we do \emph{not} use any ``integrability'' conditions like Zamolodchikov's tetrahedron equation~\cite{Zamolodchikov}.

\subsection{Remarks on Boltzmann weights}\label{ss:b}

Equal probabilities for all permitted spin configurations on a finite part~$\mathcal P$ of cubic lattice can of course be obtained from the following \emph{local Boltzmann weights}. Assign weight~$1$ to any permitted local spin configuration~$c_v$ around any vertex~$v$, and~$0$ to any prohibited~$c_v$. Then proceed as usual: Boltzmann weight~$W(c)$ for a configuration~$c$ on the whole~$\mathcal P$ is the product of weights~$w(c_v)$ of its local restrictions~$c_v$ over all vertices~$v$:
\begin{equation*}
W(c) = \prod _v w(c_v)
\end{equation*}
(this is again one or zero), and the probability of~$c$ is
\begin{equation}\label{P(c)}
P(c) = \frac{W(c)}{\sum _{\mathrm{all}\;c} W(c)}.
\end{equation}
The denominator in the rhs of~\eqref{P(c)} is called \emph{state sum}; with our Boltzmann weights, it is of course nothing but the number of all permitted~$c$.

In this paper, we consider only such ``combinatorial'' probabilities. There is, however, a strong indication that \emph{arbitrary real positive} Boltzmann weights can also, in the future, be included in our theory: it has already been shown that algebraic self-similarity can work for them as well, see~\cite[Section~VII]{block-spin}.

\subsection{The results}\label{ss:d}

The results of this paper consist, first, in the discovery of such $\mathbb F_2$-linear transformation of our spin variables that makes possible calculating correlations even for spins that are far away from each other, and second, in the actual calculations of these correlations. Below in this subsection, we
\begin{itemize}\itemsep 0pt
 \item recall the general scheme of this spin transformation,
 \item explain what difficulties occurred with its realization in~\cite{block-spin} and can be overcome with the new realization proposed in this paper,
 \item describe the kinds of correlations calculated in this paper.
\end{itemize}

\subsubsection{General facts about the linear transformation of spin variables that apply to both its version in~\cite{block-spin} and its new version introduced below in Section~\ref{s:nb}}\label{sss:b}

Let $A=(a_{ij})$ be a $3\times 3$ matrix whose entries lie in a field~$F$ of characteristic two, and consider a part~$\mathcal P$ of cubic lattice in the form of $2\times 2\times 2$ block, such as shown in Figure~\ref{f:b222},
\begin{figure}
 \centering
\begin{tikzpicture}
\draw[ thick ] (2, 0.3) -- (2, 5.7);
\draw[ thick ] (4, 0.3) -- (4, 5.7);
\draw[ thick ] (0.3, 2) -- (5.7, 2);
\draw[ thick ] (0.3, 4) -- (5.7, 4);

\draw[ thick ] (2+1.5, 0.3+1.1) -- (2+1.5, 5.7+1.1);
\draw[ thick ] (4+1.5, 0.3+1.1) -- (4+1.5, 5.7+1.1);
\draw[ thick ] (0.3+1.5, 2+1.1) -- (5.7+1.5, 2+1.1);
\draw[ thick ] (0.3+1.5, 4+1.1) -- (5.7+1.5, 4+1.1);

\filldraw (2, 2) circle (0.08);
\filldraw (2, 4) circle (0.08);
\filldraw (4, 2) circle (0.08);
\filldraw (4, 4) circle (0.08);

\draw (2, 2) node[anchor=south east] {$A\!$} ;
\draw (2, 4) node[anchor=south east] {$A\!$} ;
\draw (4, 2) node[anchor=south east] {$A\!$} ;
\draw (4, 4) node[anchor=south east] {$A\!$} ;

\filldraw (2+1.5, 2+1.1) circle (0.08);
\filldraw (2+1.5, 4+1.1) circle (0.08);
\filldraw (4+1.5, 2+1.1) circle (0.08);
\filldraw (4+1.5, 4+1.1) circle (0.08);

\draw (2+1.5, 2+1.1) node[anchor=north west] {$\!A$} ;
\draw (2+1.5, 4+1.1) node[anchor=north west] {$\!A$} ;
\draw (4+1.5, 2+1.1) node[anchor=north west] {$\!A$} ;
\draw (4+1.5, 4+1.1) node[anchor=north west] {$\!A$} ;

\draw[ thick ] ( 2-1.5, 2-1.1) -- ( 2+2*1.5, 2+2*1.1 );
\draw[ thick ] ( 2-1.5, 4-1.1) -- ( 2+2*1.5, 4+2*1.1 );
\draw[ thick ] ( 4-1.5, 2-1.1) -- ( 4+2*1.5, 2+2*1.1 );
\draw[ thick ] ( 4-1.5, 4-1.1) -- ( 4+2*1.5, 4+2*1.1 );

\draw[ thick, -latex ] (2, 0.5) -- (2, 1.15);
\draw[ thick, -latex ] (4, 0.5) -- (4, 1.15);
\draw[ thick, -latex ] (2, 6-1.25) -- (2, 6-1.15);
\draw[ thick, -latex ] (4, 6-1.25) -- (4, 6-1.15);

\draw[ thick, -latex ] (0.5, 2) -- (1.15, 2);
\draw[ thick, -latex ] (0.5, 4) -- (1.15, 4);
\draw[ thick, -latex ] (6-1.25, 2) -- (6-0.95, 2);
\draw[ thick, -latex ] (6-1.25, 4) -- (6-0.95, 4);

\draw[ thick, -latex ] (2+1.5, 0.5+1.1) -- (2+1.5, 1.15+1.1);
\draw[ thick, -latex ] (4+1.5, 0.5+1.1) -- (4+1.5, 1.15+1.1);
\draw[ thick, -latex ] (2+1.5, 6-1.25+1.1) -- (2+1.5, 6-1.15+1.1);
\draw[ thick, -latex ] (4+1.5, 6-1.25+1.1) -- (4+1.5, 6-1.15+1.1);

\draw[ thick, -latex ] (0.5+1.5, 2+1.1) -- (1.15+1.5, 2+1.1);
\draw[ thick, -latex ] (0.5+1.5, 4+1.1) -- (1.15+1.5, 4+1.1);
\draw[ thick, -latex ] (6-1.25+1.5, 2+1.1) -- (6-0.95+1.5, 2+1.1);
\draw[ thick, -latex ] (6-1.25+1.5, 4+1.1) -- (6-0.95+1.5, 4+1.1);

\draw[thick, -latex ] ( 2-1.5*0.75, 2-1.1*0.75) -- ( 2-1.5*0.55, 2-1.1*0.55);
\draw[thick, -latex ] ( 2-1.5*0.75, 4-1.1*0.75) -- ( 2-1.5*0.55, 4-1.1*0.55);
\draw[thick, -latex ] ( 4-1.5*0.75, 2-1.1*0.75) -- ( 4-1.5*0.55, 2-1.1*0.55);
\draw[thick, -latex ] ( 4-1.5*0.75, 4-1.1*0.75) -- ( 4-1.5*0.55, 4-1.1*0.55);

\draw[thick, -latex ] (2+1.5*1.35, 2+1.1*1.35) -- (2+1.5*1.55, 2+1.1*1.55);
\draw[thick, -latex ] (2+1.5*1.35, 4+1.1*1.35) -- (2+1.5*1.55, 4+1.1*1.55);
\draw[thick, -latex ] (4+1.5*1.35, 2+1.1*1.35) -- (4+1.5*1.55, 2+1.1*1.55);
\draw[thick, -latex ] (4+1.5*1.35, 4+1.1*1.35) -- (4+1.5*1.55, 4+1.1*1.55);
\end{tikzpicture}
 \caption{$2\times 2\times 2$ block. At each vertex, $3\times 3$ matrix~$A$ transforms the row of three input field~$F$ elements into the row of three output elements. The positions in a row correspond (in a fixed way) to three edge directions. Field $F=\mathbb F_2$ almost everywhere in this paper}
 \label{f:b222}
\end{figure}
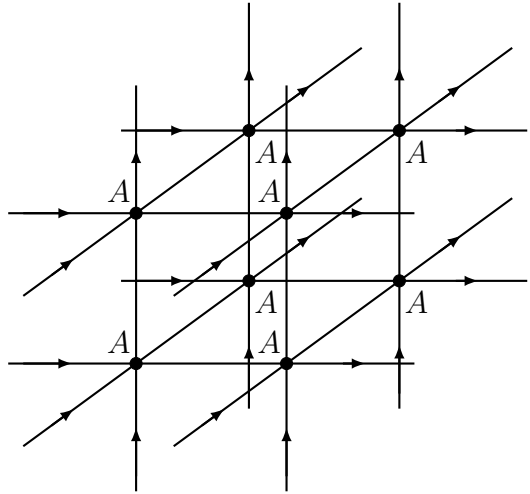
with matrix~$A$ in each of its eight vertices. Block~$\mathcal P$ represents a $12\times 12$ matrix acting in the direct sum $\mathsf V_1\oplus \mathsf V_2\oplus \mathsf V_3$ of three 4-dimen\-sional $F$-linear spaces corresponding to the three directions of coordinate axes. We write the elements of any of these~$\mathsf V_i$ (here as well as in~\cite{block-spin}) as \emph{row} 4-vectors whose components correspond to the four inputs or outputs taken in the following natural order: for an $i$-th coordinate axis, denote the (numbers of) two remaining axes as $j$ and~$k$, with $j<k$, then the inputs/outputs are numbered as follows:
\begin{gather}
 \begin{matrix} 1 & 2 \\ 3 & 4 \end{matrix} \label{1234-intro} \\*
 \begin{array}{lc}
  \text{\footnotesize direction of the $j$-th axis:} & \downarrow \\
  \text{\footnotesize direction of the $k$-th axis:} & \rightarrow
 \end{array} \nonumber
\end{gather}

It was found in~\cite{block-spin} that, under a ``general position'' condition (reproduced in this paper as formula~\eqref{Delta} below), block~$\mathcal P$ is a direct sum of \emph{four} linear operators, each acting in its own \emph{three-dimen\-sional} linear space~$W_j$, \ $j=1,\ldots,4$, whose dimensions correspond to the coordinate axes. Moreover, in matrix form, this direct sum decomposition looks as
\begin{equation}\label{AAAAT}
 \tilde A \oplus \tilde A \oplus \tilde A \oplus \bigl(\tilde A \,\bigr)^{\mathrm T},
\end{equation}
where $\tilde A$ is obtained from~$A$ by the Frobenius endomorphism of its entries: $a_{ij} \mapsto a_{ij}^2$.
 
This means that new bases can be introduced in spaces~$\mathsf V_i$, such that if we denote the new basis vectors in~$\mathsf V_i$ as $\mathsf e_1^{(i)}, \mathsf e_2^{(i)}, \mathsf e_3^{(i)}, \mathsf e_4^{(i)}$, then $W_j$ has, by definition, basis vectors $\mathsf e_j^{(1)}, \mathsf e_j^{(2)}, \mathsf e_j^{(3)}$. This way,
 \begin{equation}\label{VVV}
  \mathsf V_1\oplus \mathsf V_2\oplus \mathsf V_3 = W_1 \oplus W_2 \oplus W_3 \oplus W_4 ,
 \end{equation}
as needed.

\begin{remark}
Beware that we have slightly changed the notations and the order of basis vectors and matrices w.r.t.~\cite[Section~V.A]{block-spin}: in~\cite{block-spin}, $(\tilde A)^{\mathrm T}$ came first, and we renamed $\mathsf f^{(i)}$ into~$\mathsf e_4^{(i)}$.
\end{remark}

We typically use \emph{row} vectors (here as well as in~\cite{block-spin}, although \emph{column} vectors will also appear in this paper, starting from Subsection~\ref{ss:F}), so our matrices act on them from the \emph{right}. In particular, let $(x_j^{(i)})_{\substack{i=1,2,3 \;\; \\ j=1,2,3,4}}$ be coordinates of a vector in the space~\eqref{VVV} w.r.t.\ basis vectors~$\mathsf e_j^{(i)}$, then $\mathcal P$ acts on them as follows (decomposition~\eqref{AAAAT} in a more detailed form):
\begin{equation}\label{tildes}
 \begin{aligned}
  & \begin{pmatrix} x_j^{(1)} & x_j^{(2)} & x_j^{(3)} \end{pmatrix} \mapsto \begin{pmatrix} x_j^{(1)} & x_j^{(2)} & x_j^{(3)} \end{pmatrix} \tilde A \qquad \text{for } j=1,2,3, \\
  & \begin{pmatrix} x_4^{(1)} & x_4^{(2)} & x_4^{(3)} \end{pmatrix} \mapsto \begin{pmatrix} x_4^{(1)} & x_4^{(2)} & x_4^{(3)} \end{pmatrix} \bigl(\tilde A \, \bigr)^{\mathrm T} . \end{aligned}
\end{equation}

We can then take a $2^n\times 2^n\times 2^n$ block of matrices~$A$, and make such basis transformations iteratively. That is, we first do it in each $2\times 2\times 2$ sub-block and represent each of the resulting matrices from decomposition~\eqref{AAAAT} as a vertex. The way how $2\times 2\times 2$ sub-blocks are assembled together means for these new vertices that they are assembled into four separate $2^{n-1}\times 2^{n-1}\times 2^{n-1}$ blocks, three made of copies of~$\tilde A$, and one---of copies of~$\bigl(\tilde A \,\bigr)^{\mathrm T}$.
 
Then we proceed similarly in each of these blocks, and so on; the final result that we obtain after $n$~steps is, as \cite[Corollary~11]{block-spin} states, such bases where our initial block is represented by a direct sum of $2^{2n-1}+2^{n-1}$ matrices~$A^{(n)}$ and $2^{2n-1}+2^{n-1}$ matrices~$\bigl(A^{(n)}\bigr)^{\mathrm T}$, where $A^{(n)}$ is obtained from~$A$ by applying the Frobenius endomorphism $n$ times.

\subsubsection{Why our new version of spin transform is needed}\label{sss:w}

One essential difficulty with our old bases in $\mathsf V_1$, $\mathsf V_2$ and~$\mathsf V_3$, written out in~\cite[Formulas (21), (23) and~(24)]{block-spin}, was that it was hard, if possible at all, to obtain a general explicit expression for the resulting transform for a $2^n\times 2^n\times 2^n$ block with an arbitrary~$n$. This is why we consider it important that we can now propose another version of the same construction (described in the same general words as in Subsubsection~\ref{sss:b}, but the concrete formulas for new bases will be different from~\cite{block-spin}), where the transform for an arbitrary~$n$ is quite easily written. We call it now ``transformation from $i$-spins (initial spins) to $t$-spins (transformed spins)''.

\subsubsection{Correlations calculated in this paper}\label{sss:c}

We consider correlations for blocks $2^n\times 2^n\times 2^n$, with \emph{cyclic} (triply periodic) boundary conditions. That is, spin on each output edge is required to coincide with its input counterpart lying on the same straight line parallel to one of coordinate axes. Then we take spins on all input/output edged parallel to the \emph{first} axis (due to the cyclic boundary conditions, this is the same as if we took all edges intersecting \emph{any} fixed plane perpendicular to the first axis and not coming through the vertices), and then we are able to calculate, for a number of specific matrices~$A$, \emph{all} 2-spin, 3-spin and 4-spin correlations for these spins. Interestingly, \emph{4-spin} correlations show quite a nontrivial behavior.

\subsection{Contents of the remaining sections}\label{ss:c}

Below,
\begin{itemize}\itemsep 0pt
 \item in Section~\ref{s:nb}, we describe our new version of the self-similarity spin transform,
 \item in Section~\ref{s:sc}, we describe our approach to spin correlations, and how a Fourier transform can help in calculating them,
 \item in Section~\ref{s:t}, we explain the calculation of probabilities for the individual spins appearing as a result of the spin transform,
 \item in Section~\ref{s:A}, we present calculations for specific three-dimen\-sional eight-vertex models.
\end{itemize}

\section{New bases in thick spaces}\label{s:nb}

\subsection[Terminology: $i$-spins and $t$-spins]{Terminology: $\bm i$-spins and $\bm t$-spins}\label{ss:i,t}

Below, we work with $3\times 3$ matrices~$A$ whose entries lie in a field~$F$ of characteristic two. Until Subsection~\ref{ss:2dn}, we do not specify~$F$, although we keep in mind mainly the smallest field~$\mathbb F_2$ of two elements. We think of a $2^n\times 2^n\times 2^n$ block of copies of such matrix~$A$, in accordance with~\cite[Definition~3]{block-spin}, as a matrix that acts in the direct sum $\mathsf V_1\oplus \mathsf V_2\oplus \mathsf V_3$ of three $2^n\times 2^n$-dimen\-sional $F$-linear spaces corresponding to the three directions of coordinate axes. In~\cite{block-spin}, we used for such~$\mathsf V_j$, \ $j=1,2,3$, the name ``thick spaces'' (while an individual matrix~$A$ was thought of as acting in the direct sum of three ``thin''---one-dimen\-sional---linear spaces). We write the elements of any of these~$\mathsf V_j$ as \emph{row} vectors. The entries of these vectors corresponding to individual input edges (and to their output counterparts) are called below \emph{$i$-spins}---``initial spins'' (the \emph{order} of these entries will be specified when necessary).

There is an implied standard basis in such a linear space of row vectors, where $k$-th basis row has entry~$1\in F$ at the $k$-th place, and entries~$0\in F$ at all other places. Such basis in any of the spaces~$\mathsf V_j$ will be called \emph{$i$-basis}.

After we have applied $n$ times our basis changes, in the manner that was described in~\cite{block-spin} and recalled in Subsubsection~\ref{sss:b}, there appears a new basis in each~$\mathsf V$. We call it \emph{$t$-basis}, and the coordinates of a vector w.r.t.\ this basis---\emph{$t$-spins}.

We are now going to write out the actual form of the $t$-basis that we are proposing in this paper, concentrating mostly on the space~$\mathsf V_1$ corresponding to the first coordinate axis.

\subsection[Spin transform for a $2\times 2\times 2$ block]{Spin transform for a $\bm{2\times 2\times 2}$ block}\label{ss:2}

In this subsection, we introduce our newly discovered $t$-bases in linear spaces $\mathsf V_1, \mathsf V_2, \mathsf V_3$. They have the properties described in Subsubsection~\ref{sss:b}; below we also summarize these properties in Proposition~\ref{p:t}.

We denote entries of matrix~$A$ as~$a_{jk}\in F$, and their corresponding minors (determinants of~$A$ with $j$-th row and $k$-th column deleted) as~$m_{jk}$ (for instance, $m_{23} = a_{11}a_{32}+a_{12}a_{31}$). We assume the ``general position'' condition~\cite[(16)]{block-spin}, and write it here as
\begin{equation}\label{Delta}
 \Delta \stackrel{\mathrm{def}}{=} a_{12}a_{23}a_{31} + a_{13}a_{32}a_{21} \ne 0.
\end{equation}

\subsubsection{Spins at the input or output perpendicular to the first coordinate axis}\label{sss:a1}

Our new basis in~$\mathsf V_1$ (replacing the old basis~\cite[(21)]{block-spin}) admits the following elegant description. First, we introduce
$2\times 2$ matrices~$\mathrm G_{ij}$ as follows:
\begin{equation}\label{sfA}
\mathrm G_{ij} \stackrel{\mathrm{def}}{=} \begin{pmatrix} a_{ij} & m_{ji} \\ a_{ji} & m_{ij} \end{pmatrix}.
\end{equation}
Then, by definition, our new $t$-basis for the $2\times 2$ input or output perpendicular to the first coordinate axis---that is, basis $\mathsf e_1^{(1)}, \mathsf e_2^{(1)}, \mathsf e_3^{(1)}, \mathsf e_4^{(1)}$ in the linear space~$\mathsf V_1$, in the notations of Subsubsection~\ref{sss:b}---consists of the rows of the following $4\times 4$ matrix:
\begin{equation}\label{new21}
 \begin{pmatrix} \mathsf e_1^{(1)} \\ \mathsf e_2^{(1)} \\ \mathsf e_3^{(1)} \\ \mathsf e_4^{(1)} \end{pmatrix} = \mathrm G_{13} \otimes \mathrm G_{12}.
\end{equation}
Each of these rows is the tensor product of a row from~$\mathrm G_{13}$ by a row from~$\mathrm G_{12}$. In particular, the \emph{fourth} basis vector in~\eqref{new21},
\begin{equation}\label{21t}
 \mathsf e_4^{(1)} = \begin{pmatrix}a_{31} & m_{13}\end{pmatrix} \otimes \begin{pmatrix}a_{21} & m_{12}\end{pmatrix} = \begin{pmatrix}a_{31}a_{21} & a_{31}m_{12} & m_{13}a_{21} & m_{13}m_{12} \end{pmatrix},
\end{equation}
belongs also to the space~$W_4$ where the transposed matrix~$A^{\mathrm T}$ acts. It coincides with the \emph{first} basis vector~$\mathsf f^{(1)}$ in the basis~\cite[(21)]{block-spin} that we used earlier. It looks convenient, however, for our further calculations to have it at the fourth, rather than the first, place.

The \emph{columns} of~\eqref{new21} correspond to the $i$-spins at the input/output. Their numbering~\eqref{1234-intro} looks, in this case, as follows:
\begin{gather}
 \begin{matrix} 1 & 2 \\ 3 & 4 \\ \end{matrix} \label{1234} \\*
 \begin{array}{lc}
  \text{\footnotesize direction of the 2nd axis:} & \downarrow \\
  \text{\footnotesize direction of the 3rd axis:} & \rightarrow
 \end{array} \nonumber
\end{gather}

There is the following simple but important remark.

\begin{ir}\label{ir:f}
The fact that matrix~\eqref{new21} gives a $t$-basis can be reformulated in any of the following ways:
 \begin{itemize}\itemsep 0pt
  \item $t$-basis is obtained by applying matrix~\eqref{new21} to the column of $i$-basis vectors \emph{from the left}. This is simply because when the $i$-basis vectors (defined in Subsection~\ref{ss:i,t}) are assembled together in a column, they form nothing but the $4\times 4$ identity matrix~$\mathbbm 1_4$,
  \item for a given vector in~$\mathsf V_1$, the row of its coordinates w.r.t.\ $i$-basis is obtained from the row of its coordinates w.r.t.\ $t$-basis by action of matrix~\eqref{new21} from the right.
 \end{itemize}
\end{ir}

The tensor product form~\eqref{new21} of the $t$-basis makes it reasonable to think of the $t$-spins as also organized in a $2\times 2$ lattice like the $i$-spins in~\eqref{1234}. Moreover, we would like to introduce two \emph{ghost coordinate axes} $\alpha$ and~$\beta$ for this arrangement, so that it looks as follows:
\begin{gather}
 \begin{matrix} 1 & 2 \\ 3 & 4 \\ \end{matrix} \label{1234-ghost} \\*
 \begin{array}{lc}
  \text{\footnotesize direction of the ghost axis $\alpha$:} & \downarrow \\
  \text{\footnotesize direction of the ghost axis $\beta$:} & \rightarrow
 \end{array} \nonumber
\end{gather}

There is one more property of key importance that tensor-product basis~\eqref{new21} has: as one can see from the explicit form of matrices~\eqref{sfA}, it remains the \emph{same} if we change $A$ to~$A^{\mathrm T}$, except that its vectors go in the reverse order! This will greatly facilitate our further steps of spin transform, where we will have to work with $2\times 2\times 2$ blocks of matrices~$A^{\mathrm T}$ (we will do it for $F=\mathbb F_2$; the Frobenius endomorphism is identical in that case, and matrices~$A^{\mathrm T}$ appear according to~\eqref{AAAAT}) in parallel with blocks of matrices~$A$. In particular, the \emph{first} row in~\eqref{new21} will belong, in the case of matrices~$A^{\mathrm T}$, to the space where~$A$ (not transposed) acts.

\subsubsection{Spins at the inputs or outputs perpendicular to the other coordinate axes}\label{sss:a23}

Explicit expressions for the bases in the linear spaces $\mathsf V_2$ and~$\mathsf V_3$, corresponding to the two remaining pairs of opposite faces, can be written as follows:
\begin{equation}\label{nm}
 \begin{pmatrix} \mathsf e_2^{(2)} \\ \mathsf e_1^{(2)} \\ \mathsf e_3^{(2)} \\ \mathsf e_4^{(2)} \end{pmatrix} = \mathrm G_{23} \otimes \mathrm G_{21}, 
 \qquad \begin{pmatrix} \mathsf e_3^{(3)} \\ \mathsf e_1^{(3)} \\ \mathsf e_2^{(3)} \\ \mathsf e_4^{(3)} \end{pmatrix} = \mathrm G_{32} \otimes \mathrm G_{31}
\end{equation}
(these bases replace our old bases \cite[(23) and~(24)]{block-spin} in the same spaces).

\begin{ir}\label{ir:r}
Pay attention to the \emph{order} in which the basis vectors are written in the lhs's of the equalities~\eqref{nm}! Below in this paper, we won't deal with spins on edges parallel to coordinate axes $2$ and~$3$, but the mystery of formulas~\eqref{nm} hint at an interesting combinatorics, to be investigated in the future.
\end{ir}

\begin{proposition}\label{p:t}
$2\times 2\times 2$ block of matrices~$A$ (such as in Figure~\ref{f:b222}) acts on rows of vector coordinates in bases \eqref{new21}, \eqref{nm} according to~\eqref{tildes} (where, as we remember, $\tilde A$ is obtained from~$A$ by the Frobenius isomorphism $a_{ij} \mapsto a_{ij}^2$ of all its entries).
\end{proposition}

\begin{proof}
For the time being, it's just a direct calculation.
\end{proof}

\subsection{Binary numbering for matrix rows and columns}\label{ss:bin}

The natural basis in a linear space of rows or columns over any field consists of rows, resp.\ columns, having a unity at one position and zeros at all other positions. If the dimension (\,= length of rows/columns) of this space is~$2^n$, with an integer~$n$, then we can number these positions with $n$-binary-digit numbers using the following simple principle which we explain on the example of rows.

Each basis row is, in this case, the tensor product of $n\:$ $2$-rows, some of which are $\begin{pmatrix} 1 & 0 \end{pmatrix}$ and the others $\begin{pmatrix} 0 & 1 \end{pmatrix}$. We put binary digits in correspondence to these tensor multipliers as follows:
\begin{equation}\label{01}
  \begin{pmatrix} 1 & 0 \end{pmatrix} \mapsto 0_2, \qquad
  \begin{pmatrix} 0 & 1 \end{pmatrix} \mapsto 1_2,
\end{equation}
and take their \emph{concatenation} in the same order, obtaining thus an \emph{$n$-digit binary number} of the basis vector, as well as the position of the unity in it. Our binary numbers take, in this case, values from $\underbrace{0\ldots 0}_n {}_2$ through $\underbrace{1\ldots 1}_n {}_2$ (we can omit subscript~``2'' when no confusion is expected); note that we do not omit possible zeros at the beginning of a number.

We will also need the following slight development of this idea.

Many matrices in this paper, starting with~\eqref{new21}, and describing a spin transform, will be of sizes $2^{2n}\times 2^{2n}$, \ $n=1,2,\ldots$. Moreover, both rows and columns in such matrices, taken \emph{separately}, will correspond to a position of an $i$-spin in a $2^n\times 2^n$ face of a $2^n\times 2^n\times 2^n$ spin block, like in~\eqref{1234}, or a position of a $t$-spin in a ``ghost plane'', like in~\eqref{1234-ghost}, or there may be even a ``mixed'' case, as in~\eqref{d} below.

It turns out convenient to mark such a position in a $2^n\times 2^n$ lattice by \emph{two $n$-digit} binary numbers, each taking values in the same way as above, from $\underbrace{0\ldots 0}_{n} {}_2$ through $\underbrace{1\ldots 1}_{n} {}_2$. Considering $i$-spins on edges parallel to the first coordinate axis and intersecting a fixed plane perpendicular to them all, we can denote these numbers---``binary 2nd and 3rd coordinates''---as $b_2$ and~$b_3$, and call the pair $(b_2, b_3)$ \emph{binary address} of the corresponding $i$-spin.

Then, for a spin transform matrix whose rows correspond to $t$-spins and columns---to $i$-spins, like~\eqref{new21} or bigger matrices \eqref{4*4}, \eqref{nn} introduced below, the \emph{concatenation}~$\overline{b_2b_3}$ can be used as the binary number of its \emph{column}, taking values from $\underbrace{0\ldots 0}_{2n}$ through $\underbrace{1\ldots 1}_{2n}$. For $n=1$, we will have this way, instead of~\eqref{1234}, the following simple picture:
\begin{equation}\label{2*2}
\begin{blockarray}{ccc}
 & \scriptstyle b_3{=}0 & \scriptstyle b_3{=}1 \\
\begin{block}{c|c|c|}
 \BAhhline{&--}
\scriptstyle b_2{=}0 & 00 & 01 \\
 \BAhhline{&--}
\scriptstyle b_2{=}1 & 10 & 11 \\
 \BAhhline{&--}
\end{block}
\end{blockarray}
\end{equation}
Horizontal and vertical lines are in~\eqref{2*2}, and below in similar pictures, just for ease of perception.

Similarly, we will use \emph{binary ghost coordinates} $\alpha$ and~$\beta$ along the ghost axes of the same names, directed as in~\eqref{1234-ghost}. These will also take values from $\underbrace{0\ldots 0}_{n} {}_2$ through $\underbrace{1\ldots 1}_{n} {}_2$; concatenations~$\overline{\alpha \beta}$ can be used for numbering the \emph{rows} of a spin transform matrix, and the pair~$(\alpha, \beta)$ will be called (ghost) \emph{binary address} of a $t$-spin.

\subsection[$2^n\times 2^n\times 2^n$ blocks]{$\bm{2^n\times 2^n\times 2^n}$ blocks}\label{ss:2dn}

Below, we restrict ourself to the case of the smallest field $F=\mathbb F_2$. This will simplify, among other things, our formulas such as~\eqref{nn} below, because we will have to deal with the \emph{same} matrices $A$ and~$A^{\mathrm T}$ at every step of the $n$-step spin transform that we are going to describe, due to the fact that $\tilde A = A$ (see~\eqref{tildes}).

Consider first a $4\times 4\times 4$ block and its $4\times 4$-face perpendicular to the first axis. The analogue of binary numbering~\eqref{2*2} looks now as follows:
\begin{equation}\label{4*4-ini}
\begin{blockarray}{ccccc}
 & \scriptstyle b_3{=}00 & \scriptstyle b_3{=}01 & \scriptstyle b_3{=}10 & \scriptstyle b_3{=}11\\
\begin{block}{c|cc|cc|}
 \BAhhline{&----}
\scriptstyle b_2{=}00 & 0000 & 0001 & 0010 & 0011 \\
\scriptstyle b_2{=}01 & 0100 & 0101 & 0110 & 0111 \\
 \BAhhline{&----}
\scriptstyle b_2{=}10 & 1000 & 1001 & 1010 & 1011 \\
\scriptstyle b_2{=}11 & 1100 & 1101 & 1110 & 1111 \\
 \BAhhline{&----}
\end{block}
\end{blockarray}
\end{equation}
We divide this face into four $2\times 2$-squares, as shown in~\eqref{4*4-ini} (hopefully, it brings no trouble that the squares look like rectangles in our pictures), and, as a first step, produce transformed spins in each of these according to~\eqref{new21}. Algebraically, it means that we get new basis in the linear space of $16$-rows, consisting of the rows of matrix
\begin{equation}\label{1G1G}
 \mathbbm 1_2 \otimes \mathrm G_{13} \otimes \mathbbm 1_2 \otimes \mathrm G_{12}.
\end{equation}
The \emph{columns} of~\eqref{1G1G} correspond to $i$-spins and go in the order of 4-digit binary numbers in~\eqref{4*4-ini}; $\mathbbm 1_2$ in~\eqref{4*4-ini} is a $2\times 2$ identity matrix over~$\mathbb F_2$.

As for the \emph{rows} of~\eqref{1G1G}, there appears a \emph{mixed} binary numbering for them: they go in the order of numbers~$\overline{b_2 \alpha b_3 \beta}$, where $b_2, b_3 = 0, 1$ number the $2\times 2$ squares, while $\alpha, \beta = 0, 1$ number the positions within each $2\times 2$ square:
\begin{gather}\label{d}
\begin{blockarray}{ccccc}
 & \scriptstyle \overline{b_3 \beta}{=}00 & \scriptstyle \overline{b_3 \beta}{=}01 & \scriptstyle \overline{b_3 \beta}{=}10 & \scriptstyle \overline{b_3 \beta}{=}11\\
\begin{block}{c|cc|cc|}
 \BAhhline{&----}
\scriptstyle \overline{b_2 \alpha}{=}00 & \bm{0000} & 0001 & \bm{0010} & 0011 \\
\scriptstyle \overline{b_2 \alpha}{=}01 & 0100 & 0101 & 0110 & 0111 \\
 \BAhhline{&----}
\scriptstyle \overline{b_2 \alpha}{=}10 & \bm{1000} & 1001 & \bm{1010} & 1011 \\
\scriptstyle \overline{b_2 \alpha}{=}11 & 1100 & 1101 & 1110 & 1111 \\
 \BAhhline{&----}
\end{block}
\end{blockarray} \\*
  \begin{array}{lc}
   \text{\footnotesize direction of the 2nd axis and ghost axis $\alpha$:} & \downarrow \\
   \text{\footnotesize direction of the 3rd axis and ghost axis $\beta$:} & \rightarrow
  \end{array} \nonumber
\end{gather}

Then we do the second---and the last for the $4\times 4\times 4$ block---step: do the same transform~\eqref{new21} within every quadruple having \emph{the same ghost coordinates}, like one at the positions highlighted in bold in~\eqref{d}. Algebraically, this means applying the matrix
\begin{equation}\label{G1G1}
 \mathrm G_{13} \otimes \mathbbm 1_2 \otimes \mathrm G_{12} \otimes \mathbbm 1_2
\end{equation}
\emph{from the left} to the basis~\eqref{1G1G} obtained at the previous step (compare Important Remark~\ref{ir:f}).

The resulting transform is given by the product of \eqref{1G1G} and~\eqref{G1G1}, whose rows give the $t$-basis vectors (and $i$-spins correspond to its columns):
\begin{equation}\label{4*4}
\mathrm G_{13} \otimes \mathrm G_{13} \otimes \mathrm G_{12} \otimes \mathrm G_{12}.
\end{equation}

In the general case of a $2^n\times 2^n\times 2^n$ block orthogonal to the first coordinate axis, the algebraic description of the $n$ steps of our spin transform is done in a direct analogy with the $4\times 4\times 4$ case considered above: the $k$-th step consists in multiplying the previously obtained matrix by
\begin{equation*}
\underbrace{ \mathbbm 1_2 \otimes \cdots \otimes \mathbbm 1_2 }_{n-k} \otimes \: \mathrm G_{13} \otimes \underbrace{ \mathbbm 1_2 \otimes \cdots \otimes \mathbbm 1_2 }_{k-1}
 \;\; \otimes \;\; \underbrace{ \mathbbm 1_2 \otimes \cdots \otimes \mathbbm 1_2 }_{n-k} \otimes \: \mathrm G_{12} \otimes \underbrace{ \mathbbm 1_2 \otimes \cdots \otimes \mathbbm 1_2 }_{k-1} .
\end{equation*}
The final result is the following matrix
\begin{equation}\label{nn}
\mathbf G = \underbrace{\mathrm G_{13} \otimes \cdots \otimes \mathrm G_{13}}_n \: \otimes \: \underbrace{\mathrm G_{12} \otimes \cdots \otimes \mathrm G_{12}}_n = \mathrm G_{13}^{\otimes n} \otimes \mathrm G_{12}^{\otimes n}
\end{equation}
whose rows are the $t$-basis vectors.

\subsection{Successive spin transforms in pictures}\label{ss:tpic}

This way, we have organized new basis vectors in the $\mathbb F_2$-linear space spanned by the initial spins on the face of a cubic block perpendicular to the first coordinate axis---which we call transformed spins---in the \emph{same $2^n\times 2^n$ lattice} as we have for the initial basis vectors. We now draw some pictures, in which a \emph{square} (each of the smallest squares in each picture) corresponds to each component of an input vector in the relevant basis. The reader can think that the actual lattice points to which matrices~$A$ or~$A^{\mathrm T}$ belong lie below the centers of the corresponding squares.

First, we divide the $2^n\times 2^n$ input of a $2^n\times 2^n\times 2^n$ block, perpendicular to axis~1, in $2\times 2$ squares, and do the first step---transform to basis~\eqref{new21}---in each such square. We depict this graphically as in Figure~\ref{f:222}, showing now the relevant matrices $A$ or~$A^{\mathrm T}$ according to the fact that $A$ belongs to the three first rows of~\eqref{new21}, while~$A^{\mathrm T}$ belongs to its fourth row~\eqref{21t}.

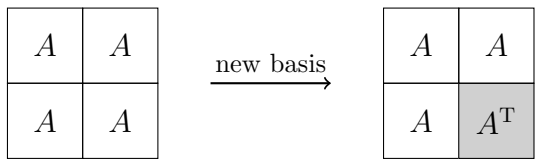
\begin{figure}
\centering
\begin{tikzpicture}
\draw (0,0) rectangle (1,1);   \draw (0.5,0.5) node[anchor=center] {$A$} ;
\draw (0,1) rectangle (1,2);   \draw (0.5,1.5) node[anchor=center] {$A$} ;
\draw (1,0) rectangle (2,1);   \draw (1.5,0.5) node[anchor=center] {$A$} ;
\draw (1,1) rectangle (2,2);   \draw (1.5,1.5) node[anchor=center] {$A$} ;
\draw[ thick, -> ] (2.7,1) -- (4.3,1);   \draw (3.5,1) node[anchor=south] {\footnotesize new basis} ;
\draw (5,0) rectangle (6,1);   \draw (5.5,0.5) node[anchor=center] {$A$} ;
\draw (5,1) rectangle (6,2);   \draw (5.5,1.5) node[anchor=center] {$A$} ;
\draw[ fill=gray!30 ] (6,0) rectangle (7,1);   \draw (6.5,0.5) node[anchor=center] {$A^{\mathrm{T}}$} ;
\draw (6,1) rectangle (7,2);   \draw (6.5,1.5) node[anchor=center] {$A$} ;
\end{tikzpicture}
\caption{LHS: Input of a $2\times 2\times 2$ block, perpendicular to axis~1. Directions of axes 2 and~3 are $\downarrow$ and~$\rightarrow$, respectively. RHS: Four separate inputs after switching to basis~\eqref{new21}. Here $\downarrow$ and~$\rightarrow$ are directions of the ghost axes $\alpha$ and~$\beta$}
\label{f:222}
\end{figure}

Then we perform more similar steps. Note that, after the first step (as well as after more successive steps), the four input vector components in any quadruple that has to undergo our transform do not lie back-to-back in our pictures. So, the transform of such a quadruple can be depicted as in Figure~\ref{f:sep}.

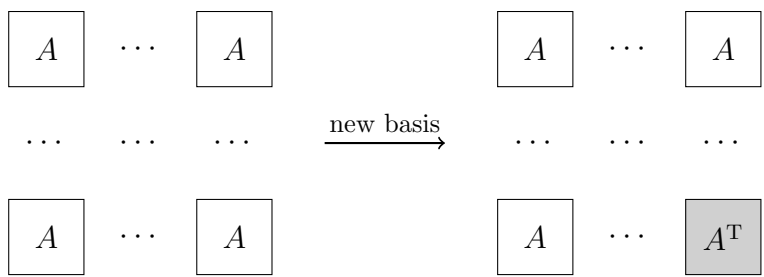
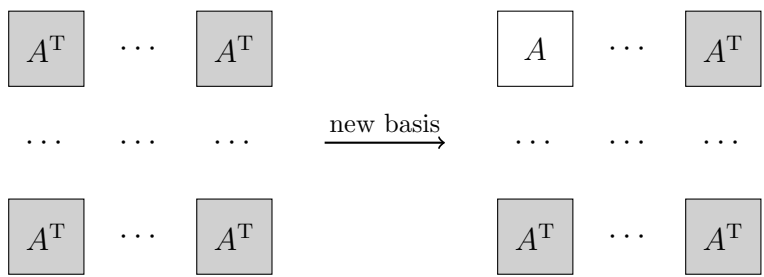
\begin{figure}
\centering
\begin{subfigure}[b]{\textwidth}
\centering
\begin{tikzpicture}
\draw (0,0) rectangle (1,1);   \draw (0.5,0.5) node[anchor=center] {$A$} ;
\draw (0,2.5) rectangle (1,3.5);   \draw (0.5,3) node[anchor=center] {$A$} ;
\draw (2.5,0) rectangle (3.5,1);   \draw (3,0.5) node[anchor=center] {$A$} ;
\draw (2.5,2.5) rectangle (3.5,3.5);   \draw (3,3) node[anchor=center] {$A$} ;
\draw[ thick, -> ] (4.2,1.75) -- (5.8,1.75);   \draw (5,1.75) node[anchor=south] {\footnotesize new basis} ;
\draw (1.5+5,0) rectangle (1.5+6,1);   \draw (7,0.5) node[anchor=center] {$A$} ;
\draw (1.5+5,2.5) rectangle (1.5+6,3.5);   \draw (7,3) node[anchor=center] {$A$} ;
\draw[ fill=gray!30 ] (3+6,0) rectangle (3+7,1);   \draw (9.5,0.5) node[anchor=center] {$A^{\mathrm{T}}$} ;
\draw (3+6,2.5) rectangle (3+7,3.5);   \draw (9.5,3) node[anchor=center] {$A$} ;
\draw (0.5,1.75) node {$\cdots$} ;
\draw (1.75,0.5) node {$\cdots$} ;
\draw (1.75,1.75) node {$\cdots$} ;
\draw (1.75,3) node {$\cdots$} ;
\draw (3,1.75) node {$\cdots$} ;
\draw (6.5+0.5,1.75) node {$\cdots$} ;
\draw (6.5+1.75,0.5) node {$\cdots$} ;
\draw (6.5+1.75,1.75) node {$\cdots$} ;
\draw (6.5+1.75,3) node {$\cdots$} ;
\draw (6.5+3,1.75) node {$\cdots$} ;
\end{tikzpicture}
 \caption{Case of matrices~$A$}
 \label{f:sepA}
\end{subfigure}

\bigskip
\bigskip
\bigskip

\begin{subfigure}[b]{\textwidth}
\centering
\begin{tikzpicture}
\draw[ fill=gray!30 ] (0,0) rectangle (1,1);   \draw (0.5,0.5) node[anchor=center] {$A^{\mathrm{T}}$} ;
\draw[ fill=gray!30 ] (0,2.5) rectangle (1,3.5);   \draw (0.5,3) node[anchor=center] {$A^{\mathrm{T}}$} ;
\draw[ fill=gray!30 ] (2.5,0) rectangle (3.5,1);   \draw (3,0.5) node[anchor=center] {$A^{\mathrm{T}}$} ;
\draw[ fill=gray!30 ] (2.5,2.5) rectangle (3.5,3.5);   \draw (3,3) node[anchor=center] {$A^{\mathrm{T}}$} ;
\draw[ thick, -> ] (4.2,1.75) -- (5.8,1.75);   \draw (5,1.75) node[anchor=south] {\footnotesize new basis} ;
\draw[ fill=gray!30 ] (1.5+5,0) rectangle (1.5+6,1);   \draw (7,0.5) node[anchor=center] {$A^{\mathrm{T}}$} ;
\draw (1.5+5,2.5) rectangle (1.5+6,3.5);   \draw (7,3) node[anchor=center] {$A$} ;
\draw[ fill=gray!30 ] (3+6,0) rectangle (3+7,1);   \draw (9.5,0.5) node[anchor=center] {$A^{\mathrm{T}}$} ;
\draw[ fill=gray!30 ] (3+6,2.5) rectangle (3+7,3.5);   \draw (9.5,3) node[anchor=center] {$A^{\mathrm{T}}$} ;
\draw (0.5,1.75) node {$\cdots$} ;
\draw (1.75,0.5) node {$\cdots$} ;
\draw (1.75,1.75) node {$\cdots$} ;
\draw (1.75,3) node {$\cdots$} ;
\draw (3,1.75) node {$\cdots$} ;
\draw (6.5+0.5,1.75) node {$\cdots$} ;
\draw (6.5+1.75,0.5) node {$\cdots$} ;
\draw (6.5+1.75,1.75) node {$\cdots$} ;
\draw (6.5+1.75,3) node {$\cdots$} ;
\draw (6.5+3,1.75) node {$\cdots$} ;
\end{tikzpicture}
 \caption{Case of matrices~$A^{\mathrm T}$}
 \label{f:sepAT}
\end{subfigure}
 \caption{General pictorial representation for a transform of four inputs}
 \label{f:sep}
\end{figure}

The second step for a $4\times 4$ square is depicted in Figure~\ref{f:44}, where the spins belonging to one quadruple of the type of Figure~\ref{f:sepA} are circled in orange.

\begin{figure}
\centering
\begin{tikzpicture}
\filldraw[gray!50] (1,0) rectangle (2,1) ;
\filldraw[gray!50] (3,0) rectangle (4,1) ;
\filldraw[gray!50] (1,2) rectangle (2,3) ;
\filldraw[gray!50] (3,2) rectangle (4,3) ;
\draw[thick] (0,0) -- (4,0);
\draw (0,1) -- (4,1);
\draw[thick] (0,2) -- (4,2);
\draw (0,3) -- (4,3);
\draw[thick] (0,4) -- (4,4);
\draw[thick] (0,0) -- (0,4);
\draw (1,0) -- (1,4);
\draw[thick] (2,0) -- (2,4);
\draw (3,0) -- (3,4);
\draw[thick] (4,0) -- (4,4);
\draw[ thick, -> ] (4.7,2) -- (6.3,2);   \draw (5.5,2) node[anchor=south] {\footnotesize new basis} ;
\filldraw[gray!50] (7+1,0) rectangle (7+2,1) ;
\filldraw[gray!50] (7+2,0) rectangle (7+3,1) ;
\filldraw[gray!50] (7+3,0) rectangle (7+4,1) ;
\filldraw[gray!50] (7+2,1) rectangle (7+3,2) ;
\filldraw[gray!50] (7+3,1) rectangle (7+4,2) ;
\filldraw[gray!50] (7+3,2) rectangle (7+4,3) ;
\draw[thick] (7+0,0) -- (7+4,0);
\draw (7+0,1) -- (7+4,1);
\draw[thick] (7+0,2) -- (7+4,2);
\draw (7+0,3) -- (7+4,3);
\draw[thick] (7+0,4) -- (7+4,4);
\draw[thick] (7+0,0) -- (7+0,4);
\draw (7+1,0) -- (7+1,4);
\draw[thick] (7+2,0) -- (7+2,4);
\draw (7+3,0) -- (7+3,4);
\draw[thick] (7+4,0) -- (7+4,4);
\draw[thick,orange] (0,1) rectangle (1,2) ;
\draw[thick,orange] (0,3) rectangle (1,4) ;
\draw[thick,orange] (2,1) rectangle (3,2) ;
\draw[thick,orange] (2,3) rectangle (3,4) ;
\draw[thick,orange] (7+0,1) rectangle (7+1,2) ;
\draw[thick,orange] (7+0,3) rectangle (7+1,4) ;
\draw[thick,orange] (7+2,1) rectangle (7+3,2) ;
\draw[thick,orange] (7+2,3) rectangle (7+3,4) ;
\end{tikzpicture}
\caption{The second step of transforms in a $4\times 4$ square}
\label{f:44}
\end{figure}
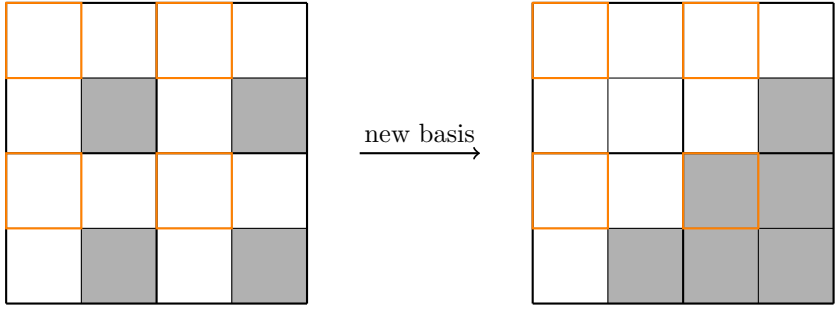

\subsection[Which $t$-spins belong to $A$ and which to~$A^{\mathrm T}$ for a general $n$]{Which $\bm t$-spins belong to $\bm A$ and which to~$\bm {A^{\mathrm T}}$ for a general $\bm n$}\label{ss:T}

We number the rows of both $\mathrm G_{13}^{\otimes n}$ and~$\mathrm G_{12}^{\otimes n}$ in~\eqref{nn} by binary numbers from ${\underbrace{0\ldots 0}_{n}}{\,}_2$ through ${\underbrace{1\ldots 1}_{n}}{\,}_2$, as described in Subsection~\ref{ss:bin}. We denote these respective binary numbers as $\alpha$ and~$\beta$, in accordance with the axes names in~\eqref{1234-ghost}.

\begin{proposition}\label{p:b}
If
\begin{equation}\label{ph}
\alpha + \beta \ge 1 \hspace{0.07em} {\underbrace{0\ldots 0}_{n}}{\,}_2 = 2^n_{10},
\end{equation}
then the corresponding row of matrix~$\mathbf G$~\eqref{nn} ($t$-basis vector) belongs to~$A^{\mathrm T}$; otherwise, it belongs to~$A$.
\end{proposition}

For $n=2$, the squares corresponding to~$A^{\mathrm T}$ are filled in with gray in the rhs of~Figure~\ref{f:44}.

\begin{proof}
Induction in~$n$.

Induction base $n=1$ is evident.

Consider induction step from~$n$ to~$n+1$. Suppose we have an address~$(\alpha, \beta)$ with $n$-digit binary $\alpha$ and~$\beta$, that is,
\begin{equation}\label{pa}
 \underbrace{0\ldots 0}_n \le \alpha, \beta < 1\underbrace{0\ldots 0}_n,
\end{equation}
as we like to write it in this context.

Let, moreover,
\begin{equation}\label{pb}
 \alpha + \beta < 1\underbrace{0\ldots 0}_n.
\end{equation}
We have then~$A$ at address~$(\alpha, \beta)$, due to the induction hypothesis.

Doing the spin transform step, we first have four \emph{combined} addresses (in the same sense as in~\eqref{d}), namely, 
\begin{equation}\label{pc}
 \begin{matrix}
  (\overline{0\alpha}, \: \overline{0\beta}) & \ldots & (\overline{0\alpha}, \: \overline{1\beta}) \\
              \vdots                      &        &               \vdots                    \\
  (\overline{1\alpha}, \: \overline{0\beta}) & \ldots & (\overline{1\alpha}, \: \overline{1\beta})
 \end{matrix} \; ,
\end{equation}
where $0$'s and~$1$'s belong to coordinate axes $x_2$ or~$x_3$.

Then, ``pure'' $t$-spin addresses appear at the same places. Treating the addresses in~\eqref{pc} as such pure addresses, and taking \eqref{pa} and~\eqref{pb} into account, we see that
\begin{equation}\label{pd}
 \begin{aligned}
  \overline{0\alpha} + \overline{0\beta} < 1\underbrace{0\ldots 0}_{n+1}, & \qquad
  \overline{0\alpha} + \overline{1\beta} < 1\underbrace{0\ldots 0}_{n+1}, \\
  \overline{1\alpha} + \overline{0\beta} < 1\underbrace{0\ldots 0}_{n+1}, & \quad \text{but} \quad
  \overline{1\alpha} + \overline{1\beta} \ge 1\underbrace{0\ldots 0}_{n+1}.
 \end{aligned}
\end{equation}
Inequalities~\eqref{pd} mean that three~$A$'s and one~$A^{\mathrm T}$ appear exactly where desired.

If, on the contrary, we begin the induction step with
\begin{equation}\label{pf}
 \alpha + \beta \ge 1\underbrace{0\ldots 0}_n,
\end{equation}
that is, address $(\alpha, \beta)$ belonging to~$A^{\mathrm T}$, then it follows from \eqref{pa} and~\eqref{pf} that
\begin{equation}\label{pg}
 \begin{aligned}
  \overline{0\alpha} + \overline{0\beta} < 1\underbrace{0\ldots 0}_{n+1}, & \quad \text{but} \quad
  \overline{0\alpha} + \overline{1\beta} \ge 1\underbrace{0\ldots 0}_{n+1}, \\
  \overline{1\alpha} + \overline{0\beta} \ge 1\underbrace{0\ldots 0}_{n+1}, & \qquad
  \overline{1\alpha} + \overline{1\beta} \ge 1\underbrace{0\ldots 0}_{n+1}.
 \end{aligned}
\end{equation}
Inequalities~\eqref{pg} mean, again, that the matrices appear where desired, and this time they are one~$A$ and three~$A^{\mathrm T}$\,'s.
\end{proof}

In terms of our pictures, the squares filled with gray in the rhs of Figure~\ref{f:44} correspond to $t$-spin addresses satisfying~\eqref{ph} for $n=2$.

\subsection{Inverting the spin transform}\label{ss:ti}

It turns out that condition~\eqref{Delta} ensures the \emph{invertibility} of our spin transform. Moreover, the inversion is easily done explicitly.

The inverse of matrix~$\mathrm G_{ij}$~\eqref{sfA} looks as follows:
\begin{equation*}
\mathrm B_{ij} \stackrel{\mathrm{def}}{=} \mathrm G_{ij}^{-1} = \frac{1}{\Delta} \begin{pmatrix} m_{ij} & m_{ji} \\ a_{ji} & a_{ij} \end{pmatrix},
\end{equation*}
where $\Delta$ from~\eqref{Delta} appears because
\begin{equation*}
\Delta = \det \mathrm G_{ij}
\end{equation*}
for any $1\le i,j\le 3$, \ $i\ne j$.

Hence, the inverse of matrix~$\mathbf G$~\eqref{nn} is
\begin{equation}\label{BB}
\mathbf B \stackrel{\mathrm{def}}{=} \mathrm G^{-1} = \underbrace{\mathrm B_{13} \otimes \cdots \otimes \mathrm B_{13}}_n \: \otimes \: \underbrace{\mathrm B_{12} \otimes \cdots \otimes \mathrm B_{12}}_n = \mathrm B_{13}^{\otimes n} \otimes \mathrm B_{12}^{\otimes n} .
\end{equation}
Matrix~$\mathbf B$ transforms, of course, a row of $i$-spins into the corresponding row of $t$-spins. Interestingly, however, it will actually be of use in this paper in a somewhat different situation, namely when it acts on \emph{column} vectors, in the proof of Theorem~\ref{th:theorem} below! The invertibility of~$\mathbf G$ will also be important for the proof of Proposition~\ref{p:duals}.

\section{Spin correlations: Generalities}\label{s:sc}

\subsection{Spin configurations and probabilities}\label{ss:cp}

A permitted spin configuration on a $2^n\times 2^n\times 2^n$ block is determined by $3\times 2^n\times 2^n$ \emph{input} spins on its three input planes, orthogonal to the coordinate axes 1, 2 and~3. Moreover, we can impose some \emph{boundary conditions} on spin configurations; in this case, we use the word ``permitted'' for the configurations that are compatible with these conditions. Specifically, in this paper, we will work with \emph{triply-periodic} conditions: an input edge is identified with an output edge if they both lie on the same line parallel to one of the coordinate axes. We also call these conditions simply \emph{cyclic}. We assume (as we already said in the \hyperref[s:i]{Introduction}) that any permitted spin configuration has Boltzmann weight~$1\in \mathbb R$, while all other configurations have Boltzmann weight~$0\in \mathbb R$.

We will be able to calculate the probability of a ``random event'' such as ``some chosen spins have some chosen values'' if we manage to express this event in terms of $t$-spins. This is because of the following proposition.

\begin{proposition}\label{p:i}
If cyclic boundary conditions are imposed on a $2^n\times 2^n\times 2^n$ block, then $t$-spins are \emph{independent random variables}.
\end{proposition}

\begin{proof}
Cyclic boundary conditions for $i$-spins remain the same cyclic boundary conditions after passing to $t$-spins, because the transition to $t$-spins is given by the same invertible $\mathbb F_2$-linear operators on each input face and its output counterpart. But cyclic boundary conditions act on each $t$-spin \emph{separately}: they mean that \emph{its} input must coincide with \emph{its} output (hence, both the input and output must lie in the \emph{eigenspace} of the corresponding operator $A$ or~$A^{\mathrm T}$, with the eigenvalue~$1\in \mathbb F_2$. We discuss it in more detail in Subsection~\ref{ss:pe}).

Our ``classical'' probability of the event ``all $t$-spins have taken some prescribed values'' is the ratio
\begin{equation}\label{pv}
\frac{\text{\small number of permitted $t$-spin configurations with the prescribed values}}{\text{\small number of all permitted $t$-spin configurations}},
\end{equation}
The numerator in~\eqref{pv} is~$1$ if all the prescribed values are actually permitted for all $t$-spins, and~$0$ otherwise. It is hence the product of the corresponding zeros and ones for the individual $t$-spins. The denominator is a product as well---the product of the numbers of permitted configurations for individual $t$-spins (which are the cardinalities of the eigenspaces mentioned above in parentheses). Hence, the whole probability~\eqref{pv} is the product of the probabilities for individual $t$-spins.
\end{proof}

\begin{remark}
There are also some other kinds of boundary conditions that retain their form after passing to $t$-spins, such as:
 \begin{itemize}\itemsep 0pt
  \item all spins on a given face of the block are zero,
  \item all spins on a given face are free.
 \end{itemize}
\end{remark}

We will be considering the probabilities of the events ``spins on $k$~given edges have taken zero (\,$0 \in \mathbb F_2$\,) values'', where $k=2,3$ or~$4$, for a $2^n\times 2^n\times 2^n$ block and input (\,=\,output, because of the cyclic boundary conditions) edges parallel to the first coordinate axis.
These $k$ zero conditions specify a hyperplane of codimension~$k$ in the vector space $\mathbb F_2^{3\times 4^n}$ of all input spin configurations (having, as we remember, Boltzmann weights either $0$ or~$1\in \mathbb R$). A big sum over this hyperplane becomes, however, up to a normalization factor, just a sum over a $k$-dimensional $\mathbb F_2$-plane if we introduce into play a \emph{multidimensional discrete Fourier transform} for $\mathbb R$-valued functions on an $\mathbb F_2$-linear space~$\mathbb F_2^{3\times 4^n}$, as specified in Proposition~\ref{p:sum-kw} in the next Subsection~\ref{ss:F}.

\subsection{Multidimensional discrete Fourier transform}\label{ss:F}

We will be using multidimensional discrete Fourier transform for $\mathbb R$-valued functions on a finite-dimen\-sional $\mathbb F_2$-linear space~$L$ in the following form. We write the elements of~$L$ as \emph{row} vectors, and denote them as~$\vec x$. For an $\mathbb R$-valued function~$f$ on~$L$, its Fourier transform is the following $\mathbb R$-valued function on the \emph{dual} space $L^*$ consisting of \emph{column} vectors~$\vec w$:
\begin{equation}\label{Fourier}
F(\vec w) \stackrel{\mathrm{def}}{=} \sum_{\vec x \in L} f(\vec x) (-1)^{\vec x \vec w}.
\end{equation}
where $\vec x \vec w \in \mathbb F_2$ is the usual product of a row and a column; $(-1)^{\vec x \vec w}$ is a well defined real number $\pm 1$.

We will be using the properties of transform~\eqref{Fourier} formulated below as Propositions \ref{p:sum-kw} and~\ref{p:prod}.

\begin{proposition}\label{p:sum-kw}
Sum of values~$f(\vec x)$ for all $\vec x$ \emph{orthogonal} to a $k$-dimen\-sional $\mathbb F_2$-linear subspace~$K \subset L^*$ of columns~$\vec w$ is expressed by means of transform~\eqref{Fourier} as follows:
 \begin{equation}\label{sum-kw}
  \sum_{\vec x \; \perp \; K} f(\vec x) = \frac{\sum _{\vec w\in K} F(\vec w)}{2^k}.
 \end{equation}
\end{proposition}

\begin{proof}
Indeed, for a given~$\vec x$, the coefficient at~$f(\vec x)$ in $\sum _{\vec w\in K} F(\vec w)$ is, after substituting~\eqref{Fourier},
\begin{equation}\label{K}
\sum _{\vec w\in K} (-1)^{\vec x \vec w}.
\end{equation}
For $\vec x \perp K$, this is the sum of unities over~$K$, that is,~$2^k$.

If, on the contrary, there is a vector $\vec a\in K$ such that $(-1)^{\vec x \vec a} = -1$, then $K$ splits into pairs of the type $(\vec w, \vec w + \vec a)$, and for each such pair $(-1)^{\vec x \vec w} + (-1)^{\vec x (\vec w + \vec a)} = 0$, hence the whole sum~\eqref{K} is~$0$.
\end{proof}

It may make sense to write out a few particular cases of~\eqref{sum-kw}:
\begin{itemize}\itemsep 0pt
 \item sum of all values~$f(\vec x)$:
  \begin{equation}\label{sum-all}
   \sum_{\vec x \in \mathbb L} f(\vec x) = F(\vec 0),
  \end{equation}
 \item sum of values~$f(\vec x)$ for all $\vec x$ orthogonal to a nonzero column~$\vec w$:
  \begin{equation}\label{sum-1w}
   \sum_{\vec x \perp \vec w} f(\vec x) = \frac{F(\vec 0) + F(\vec w)}{2},
  \end{equation}
 \item sum of values~$f(\vec x)$ for all $\vec x$ orthogonal to two linearly independent columns $\vec w_1$ and~$\vec w_2$:
  \begin{equation}\label{sum-2w}
   \sum_{\vec x \; \perp \; \vec w_1 \mathrm{\, and \,} \vec w_2} f(\vec x) = \frac{F(\vec 0) + F(\vec w_1) + F(\vec w_2) + F(\vec w_1 + \vec w_2)}{4}.
  \end{equation}
\end{itemize}

\begin{proposition}\label{p:prod}
Let $L$ be represented as a \emph{direct sum}: $L=\bigoplus_{i=1}^k L_i$. For a vector $\vec x\in L$, we write as $\vec x = \sum _{i=1}^k \vec x_i$, \ $\vec x_i \in L_i$, its corresponding decomposition into a sum. Let $f$ be the product of functions~$\varphi _i$ defined each on its space~$L_i$:
 \begin{equation}\label{fd}
   f(\vec x) = \varphi _1(\vec x_1) \cdots \varphi _k(\vec x_k).
 \end{equation}
  Then, the Fourier transform~$F$ of\/~$f$ is also a product:
 \begin{equation}\label{Fp}
   F(\vec w) = \Phi _1(\vec w_1) \cdots \Phi _k(\vec w_k),
 \end{equation}
  where the decomposition $\vec w = \vec w_1 + \dots + \vec w_k$ corresponds to the \emph{dual decomposition} of the dual space: $L^*=M_1 \oplus \dots \oplus M_k$, each $M_i$ being the space orthogonal to all~$L_j$ with $j\ne i$, and
 \begin{equation}\label{Phi_i}
   \Phi _i(\vec w_i) = \sum_{\vec x_i \in L_i} \varphi _i (\vec x_i) (-1)^{\vec x_i \vec w_i}.
 \end{equation}
\end{proposition}

\begin{proof}
Under the proposition conditions, each summand in~\eqref{Fourier} becomes a product, namely, $f(\vec x)$ according to~\eqref{fd}, and $(-1)^{\vec x\vec w} = (-1)^{\vec x_1\vec w_1} \cdots (-1)^{\vec x_k\vec w_k}$. Moreover, the summation over $\vec x\in L$ is equivalent to the multiple summation over vectors~$\vec x_i$, each taking independently all values in its space~$L_i$. That is, \eqref{Fourier} becomes
\begin{equation*}
F(\vec w) = \sum _{\vec x_1\in L_1} \cdots \sum _{\vec x_k\in L_k} \varphi _1(\vec x_1) \cdots \varphi _k(\vec x_k) \, (-1)^{\vec x_1\vec w_1} \cdots (-1)^{\vec x_k\vec w_k},
\end{equation*}
which is the same as \eqref{Fp} and~\eqref{Phi_i}.
\end{proof}

It follows from Proposition~\ref{p:prod} that, for independent random variables such as our $t$-spins, the Fourier transform of their joint probability distribution is also the product of their individual Fourier transforms.

\subsection[Obtaining $t$-spin duals from $i$-spin duals]{Obtaining $\bm t$-spin duals from $\bm i$-spin duals}\label{ss:dt}

The Fourier transform of an $i$- or $t$-spin probability distribution is a function on the space of \emph{column} vectors---dual space to the space of spins. And a column~$\vec w$ of $t$-spin duals is obtained from its corresponding column~$\vec z$ of $i$-spin duals by the action of the \emph{same} matrix~$\mathbf G$~\eqref{nn} that transforms a \emph{row} of $t$-spins into its corresponding row of $i$-spins, as the following Proposition~\ref{p:duals} states. The difference is, however, that now matrix~\eqref{nn} acts on columns, and hence from the left!

According to Subsection~\ref{ss:2dn}, for a given a row~$\vec x$ of $t$-spins, its corresponding row~$\vec y$ of $i$-spins is
\[
\vec y = \vec x \mathbf G ,
\]
where $\mathbf G$~\eqref{nn} is the matrix whose rows are $t$-basis vectors.

\begin{proposition}\label{p:duals}
If we have an $\mathbb R$-valued function~$h$ on vectors~$\vec y$, with Fourier transform~$H$, then the Fourier transform~$F$ of function~$f$ defined as $f(\vec x) = h(\vec y)$ satisfies $F(\vec w) = H(\vec z)$, if
\[
\vec w = \mathbf G \vec z .
\]
\end{proposition}

\begin{proof}
\[
H(\vec z) = \sum _{\vec y} h(\vec y) (-1)^{\vec y \vec z} = \sum _{\vec x} h(\vec x \mathbf G) (-1)^{\vec x \mathbf G \vec z} = \sum _{\vec x} f(\vec x) (-1)^{\vec x (\mathbf G \vec z)} = F(\mathbf G \vec z).
\]
Here the second equality, where we pass from summation over~$\vec y$ to  summation over~$\vec x$, uses the invertibility of~$\mathbf G$, which we have shown in Subsection~\ref{ss:ti}.
\end{proof}

$t$-Spin duals obtained from the following \emph{$i$-spin duals for individual spins} will be of particular interest when we come to calculations of spin correlations in Section~\ref{s:A}. Consider linear space~$\mathsf V_1^*$ dual to~$\mathsf V_1$, and consider the basis in~$\mathsf V_1^*$ dual to $i$-basis in~$\mathsf V_1$. The element of this dual basis \emph{dual to a chosen $i$-basis vector} will be called \emph{dual column} to that vector, or to the $i$-spin corresponding to that vector.

\section[Probabilities of individual $t$-spins, and Fourier transform of a probability function for a $t$-spin dual column]{Probabilities of individual $\bm t$-spins, and Fourier transform of a probability function for a $\bm t$-spin dual column}\label{s:t}

\subsection[Eigenspaces of $A$ or $A^{\mathrm T}$ and input vector probabilities]{Eigenspaces of $\bm A$ or $\bm{A^{\mathrm T}}$ and input vector probabilities}\label{ss:pe}

After the transition to $t$-spins, the $\mathbb F_2$-linear operator represented by our $2^n\times 2^n\times 2^n$ block becomes, in matrix form, a \emph{direct sum} of $2^{2n}$ matrices. Namely, we have $2^{2n-1} + 2^{n-1}$ matrices~$A$, and $2^{2n-1} - 2^{n-1}$ matrices~$A^{\mathrm T}$, according to~\cite[Corollary~11]{block-spin} (and taking into account that the Frobenius transform on~$\mathbb F_2$ just leaves its elements where they were). Denote $p(x_1,x_2,x_3)$ the probability of an input vector $\begin{pmatrix}x_1 & x_2 & x_3\end{pmatrix} \in \mathbb F_2^3$ for a single~$A$. With our boundary conditions, it must be also an output vector:
\begin{equation*}
\begin{pmatrix}x_1 & x_2 & x_3\end{pmatrix} A = \begin{pmatrix}x_1 & x_2 & x_3\end{pmatrix},
\end{equation*}
hence, a permitted spin configuration occurs provided $\begin{pmatrix}x_1 & x_2 & x_3\end{pmatrix}$ lies in the \emph{eigenspace}~$E$ of~$A$ corresponding to eigenvalue~$1$. As we are considering the ``classical probability'' case where all permitted configurations have equal probabilities, and denoting $\dim E = d$, we have for an individual matrix~$A$:
\begin{equation}\label{pe}
p(x_1,x_2,x_3) = \left\{ \begin{array}{cl} \frac{1}{2^d} & \text{if } \begin{pmatrix}x_1 & x_2 & x_3\end{pmatrix} \in E, \\[0.3ex] 0 & \text{otherwise.} \end{array} \right.
\end{equation}

The same holds for~$A^{\mathrm T}$; in this case, we add a prime to the letter denoting a probability, e.g.,~$p'(x_1,x_2,x_3)$.

\subsection{Probabilities of a first spin value, at a given binary address}\label{ss:p1}

Denote $q(x_1)$ the probability of the event ``the first coordinate of an input/output vector of matrix~$A$ took a given value~$x_1$'', in the context of Subsection~\ref{ss:pe}. It is, of course, the sum over all possible values for $x_2$ and~$x_3$:
\begin{equation}\label{p1}
 \begin{aligned}
  q(0) = p(0,0,0) + p(0,0,1) + p(0,1,0) + p(0,1,1), \\
  q(1) = p(1,0,0) + p(1,0,1) + p(1,1,0) + p(1,1,1).
 \end{aligned}
\end{equation}
Similarly, probabilities~$q'(x_1)$ are defined for~$A^{\mathrm T}$.

Recall that Proposition~\ref{p:b} determines whether a specific spin~$x_1$ with a binary address $(\alpha, \beta)$ belongs to matrix~$A$ or~$A^{\mathrm T}$.

\subsection[Fourier transform of a probability function: value for a given $t$-spin dual column]{Fourier transform of a probability function: value for a given $\bm t$-spin dual column}\label{ss:v}

Denote $Q$ and~$Q'$ the Fourier transforms~\eqref{Fourier} of the real-valued functions $q$ and~$q'$ defined in Subsection~\ref{ss:p1}. As the domain of these functions is just the one-dimen\-sional space~$\mathbb F_2$, this simply means that
\begin{equation}\label{Q}
 Q(0) = q(0) + q(1), \qquad Q(1) = q(0) - q(1),
\end{equation}
and similarly for~$Q'$.

Denote now $f$ the following probability function:  $f(\vec x)$ is the probability of a $2^{2n}$-row~$\vec x$ of $t$-spins (obtained as a result of the spin transform of $i$-spins at the edges perpendicular to the first coordinate axis), and denote~$F$ its Fourier transform defined on columns~$\vec w$ of $t$-spin duals. $F(\vec w)$ is, for a given~$\vec w = \begin{pmatrix} w_1 & \dots & w_{2^{2n}} \end{pmatrix}^{\mathrm T}$, the product over all individual $t$-spin duals $Q(w_j)$ or~$Q'(w_j)$ in~$\vec w$, according to Propositions \ref{p:i} and~\ref{p:prod}.

For zero duals $w_j=0$, the multiplier $Q(0)$ or~$Q'(0)$ is unity, so what must really be done is take the product of values $Q(1)$ or~$Q'(1)$ over the binary addresses corresponding to entries $1\in \mathbb F_2$ in the column~$\vec w$.

\section[Calculations for specific matrices~$A$]{Calculations for specific matrices~$\bm A$}\label{s:A}

We are going to calculate one-, two-. three-, and four-spin probabilities for spins lying on edges parallel to the first coordinate axis, with an additional condition that all these edges intersect one fixed plane perpendicular to that axis. To be more exact, we will calculate the probabilities that the $i$-spins on those edges all take value~$0\in \mathbb F_2$ (as there are just two values for any spin, there is no problem then to calculate the probability of those spin taking any other values).  We impose \emph{cyclic boundary conditions}, so we can assume that the mentioned edges are input (\,=\,output\,), without loss of generality.

\subsection[Probabilities for a specific $A$ and its transpose, and their Fourier transforms]{Probabilities for a specific $\bm A$ and its transpose, and their Fourier transforms}\label{ss:pf}

We begin with choosing the following matrix~$A$:
\begin{equation}\label{cA}
A = (a_{ij}) = \begin{pmatrix} 0 & 1 & 1 \\ 0 & 0 & 1 \\ 1 & 0 & 1 \end{pmatrix}.
\end{equation}
We impose cyclic boundary conditions on~$A$~\eqref{cA}. As $A$ has exactly one row eigenvector, namely $\begin{pmatrix} 1 & 1 & 1 \end{pmatrix}$, we obtain, according to~\eqref{pe}, the following probabilities:
\begin{equation*}
p(0,0,0) = p(1,1,1) = \frac{1}{2}\,,\quad \text{other probabilities are zero},
\end{equation*}
which gives for the~$q$'s, according to~\eqref{p1},
\begin{equation*}
q(0) = q(1) = \frac{1}{2} .
\end{equation*}
Hence, the Fourier transform~\eqref{Q} is:
\begin{equation}\label{QA}
 Q(0) = 1, \qquad Q(1) = 0.
\end{equation}

For $A^{\mathrm T}$, we have the eigenvector $\begin{pmatrix} 0 & 1 & 1 \end{pmatrix}$, and obtain the following probabilities:
\[
p'(0,0,0) = p'(0,1,1) = \frac{1}{2}\,,\quad \text{other probabilities are zero},
\]
which gives for the~$q'$'s
\[
q'(0) = 1, \qquad q'(1) = 0 .
\]
In this case, the Fourier transform is:
\begin{equation}\label{QAT}
 Q'(0) = Q'(1) = 1 .
\end{equation}

\subsection[Matrix giving $t$-spin duals for the chosen~$A$]{Matrix giving $\bm t$-spin duals for the chosen~$\bm A$}\label{ss:h}

The tensor multipliers in matrix~$\mathbf G$~\eqref{nn} of the spin transform in the case of~$A$ given by~\eqref{cA} are, according to~\eqref{sfA},
\begin{equation}\label{GG}
\mathrm G_{13} = \begin{pmatrix}1 & 1\\ 1 & 0\end{pmatrix},\qquad \mathrm G_{12} = \begin{pmatrix}1 & 1\\ 0 & 1\end{pmatrix}.
\end{equation}
Each of them has thus one column~$\begin{pmatrix} 1 \\ 0 \end{pmatrix}$ and one column~$\begin{pmatrix} 1 \\ 1 \end{pmatrix}$, and this will turn out to be of key importance in our calculations.

\subsection{One-spin function}\label{ss:1-p}

Recall that we take a spin on an edge parallel to axis~$1$. Due to our cyclic boundary conditions, all such edges are on an equal footing, so it is enough to take the \emph{input} edge with the binary address most convenient for our further calculations, namely
\begin{equation}\label{bb10}
 (b_2, b_3) = ( \, \underbrace{1\ldots 1}_n\,, \: \underbrace{0\ldots 0}_n \, ).
\end{equation}
The $i$-spin dual column (see Subsection~\ref{ss:dt}) corresponding to this edge is
\begin{equation*}
\vec r =
 \underbrace{\begin{pmatrix} 0 \\ 1 \end{pmatrix} \otimes \cdots \otimes \begin{pmatrix} 0 \\ 1 \end{pmatrix}}_n \, \otimes \,
 \underbrace{\begin{pmatrix} 1 \\ 0 \end{pmatrix} \otimes \cdots \otimes \begin{pmatrix} 1 \\ 0 \end{pmatrix}}_n ,
\end{equation*}
hence, the corresponding column of $t$-spin duals is, according to Proposition~\ref{p:duals}, and due to the form of matrices~\eqref{GG},
\begin{equation}\label{s1}
 \vec w = \mathbf G \, \vec r =
 \underbrace{\begin{pmatrix} 1 \\ 0 \end{pmatrix} \otimes \cdots \otimes \begin{pmatrix} 1 \\ 0 \end{pmatrix}}_n \, \otimes \,
 \underbrace{\begin{pmatrix} 1 \\ 0 \end{pmatrix} \otimes \cdots \otimes \begin{pmatrix} 1 \\ 0 \end{pmatrix}}_n .
\end{equation}
That is, just one nonzero $t$-spin appears in~$\vec w$, with binary address
\begin{equation}\label{ab00}
(\alpha,\beta) = (\underbrace{0\ldots 0}_n, \: \underbrace{0\ldots 0}_n),
\end{equation}
where matrix~$A$ (not transposed, according to Proposition~\ref{p:b}) stands, with the corresponding value $Q(1)=0$, according to~\eqref{QA}. Hence, $F(\vec w)$ in~\eqref{sum-1w} is also zero (because $F(\vec w)$ is a product, as in~\ref{Fp}, and now it is a product of the zero at the address~\eqref{ab00} and unities $Q(0)=1$ at all other addresses).

Concerning $F(\vec 0)=1$, it is, according to~\eqref{sum-all}, the sum of all probabilities~$f(\vec x)$, hence $F(\vec 0)=1$. We get thus finally that the probability of our spin being zero is~$\frac{1}{2}$.

\subsection{Two-spin correlations}\label{ss:2-p}

Here we take again one spin with address~\eqref{bb10}, with its column~\eqref{s1} of $t$-spins duals. For calculations in this subsection, we rename that column from~$\vec w$ to~$\vec w_1$.

For the second spin, there arises a column, we call it~$\vec w_2$, not coinciding with~\eqref{s1}, and being, due to~\eqref{GG}, the tensor product of some 2-columns $\begin{pmatrix} 1 \\ 1 \end{pmatrix}$ and some~$\begin{pmatrix} 1 \\ 0 \end{pmatrix}$. The binary addresses of unities in~$\vec w_2$ are, according to the numeration introduced in Subsection~\ref{ss:bin}, those having \emph{any} digit at the same positions as $\begin{pmatrix} 1 \\ 1 \end{pmatrix}$ in the tensor product, and \emph{zero} at the positions of~$\begin{pmatrix} 1 \\ 0 \end{pmatrix}$.

Applying formula~\eqref{sum-2w}, we first note that $F(\vec 0)$ and $F(\vec w_1) = F(\vec w)$ are the same as in Subsection~\ref{ss:1-p}, and $F(\vec w_2)$ is the same as~$F(\vec w_1)$  due to translational invariance. That is,
\begin{equation}\label{0w1w2}
F(\vec 0)=1,\qquad F(\vec w_1)=F(\vec w_2)=0.
\end{equation}
Consider now column $\vec w_1+\vec w_2$. Due to presence of at least one multiplier $\begin{pmatrix} 1 \\ 1 \end{pmatrix}$ in~$\vec w_2$, it contains necessarily a nonzero entry at the address $(\alpha, \beta)$ with either $\alpha$ all zeros but $\beta$ not, or vice versa. This address corresponds to~$A$ (not transposed), according to Proposition~\ref{p:b}, and this brings about the multiplier~$Q(1)=0$, according to~\eqref{QA}.

So, \emph{all} terms except~$F(\vec 0)$ in the numerator of the rhs of~\eqref{sum-2w} vanish, hence, the two-spin correlation function is~$\frac{1}{4}$. This means that there is no correlation between a pair of spins in our case, as long as we consider only this pair.

\subsection{Three-spin correlations}\label{ss:3-p}

Here, formula~\eqref{sum-kw} applies, with $k=3$, and $\mathbb F_2$-linear space~$K$ having basis vectors $\vec w_1$, $\vec w_2$ and~$\vec w_3$. Then, all the summands in its rhs that have the sum of either one or two~$\vec w_j$ in the argument of~$F$ are already zeros, due to Subsections \ref{ss:1-p} and~\ref{ss:2-p} and translational invariance.

Concerning $F(\vec w_1 + \vec w_2 + \vec w_3)$, we note that all these vectors~$\vec w_j$ have a unity at the address $(\alpha, \beta)$ with \emph{all zero} coordinates, which address corresponds, as we remember, to~$A$. Three unities added together give again a unity, so this address gives factor~$Q(1)=0$, from where $F(\vec w_1 + \vec w_2 + \vec w_3) = 0$.

Hence, the three-spin function is~$\frac{1}{8}$, which means that any three spins of the kind specified in the beginning of this section are also independent.

\subsection{Four-spin correlations}\label{ss:4-p}

\begin{lemma}\label{l:m}
Let $n=1,2,\ldots$ be a natural number, and consider $\mathbb F_2$-linear space~$\mathbb F_2^{2^n}$, whose elements we write here as \emph{column} vectors. Consider four such vectors $\vec a, \vec b, \vec c, \vec d$ of the following special structure: each of them is a tensor product of $n$ $2$-columns, each of these latter being either $\begin{pmatrix}1\\ 0\end{pmatrix}$ or~$\begin{pmatrix}1\\ 1\end{pmatrix}$ (where $0,1\in \mathbb F_2$ of course). Then, condition
\begin{equation}\label{abcd}
\vec a + \vec b + \vec c + \vec d = \vec 0
\end{equation}
implies that there are two pairs of equal vectors among $\vec a, \vec b, \vec c, \vec d$ (for instance, $\vec a = \vec b$ and $\vec c = \vec d$).
\end{lemma}

\begin{proof}
For each of $\vec a, \vec b, \vec c, \vec d$, consider the set of $n$-binary-digit addresses of its entries~$1\in \mathbb F_2$ (defined in the same way as in~\eqref{01}, only changing rows to columns). Denote these sets $\mathsf a, \mathsf b, \mathsf c, \mathsf d$. It follows from~\eqref{abcd} that \emph{any} binary address belongs to an \emph{even} number of them.

In particular, the \emph{greatest} binary number~$h$ in the union $\mathsf a \cup \mathsf b \cup \mathsf c \cup \mathsf d$ belongs to either two or all four of $\mathsf a, \mathsf b, \mathsf c, \mathsf d$. If, for instance, $h\in \mathsf a$ and~$h\in \mathsf b$, and taking into account that the unities in~$h$ stand exactly at the same places where multipliers~$\begin{pmatrix}1\\ 1\end{pmatrix}$ occur in the corresponding vector. we see that $\vec a = \vec b$ and consequently $\vec c = \vec d$.
\end{proof}

We consider four $t$-spin duals $\vec w_j$, \ $j=1,2,3,4$. They all are tensor products
\begin{equation}\label{wuv}
 \vec w_j = \vec u_j \otimes \vec v_j
\end{equation}
of two $2^n$-column vectors, simply because $\vec w_j = \mathbf G \cdot (i\text{-dual})$, according to Proposition~\ref{p:duals}, with $\mathbf G$ having the form~\eqref{nn}, and every $i$-dual being a tensor product of 2-columns $\begin{pmatrix}1\\ 0\end{pmatrix}$ or~$\begin{pmatrix}0\\ 1\end{pmatrix}$.

\begin{lemma}\label{l:l}
Denote~$\mathsf h$ the set of binary addresses (this time they are $2n$-digit) of unities in the sum $\vec w_1 + \vec w_1 + \vec w_2 + \vec w_3 + \vec w_4$. A \emph{necessary} condition for~$\mathsf h$ to have no addresses corresponding to~$A$ is that there are two pair of equal vectors among both the~$u_j$ and~$v_j$.
\end{lemma}

\begin{proof}
In our ``ghost'' coordinates $(\alpha, \beta)$, both the first row (\,$\alpha=0$\,) and first column (\,$\beta=0$\,) correspond only to~$A$, not~$A^{\mathrm T}$. Concerning, for instance, the first row, we see that it must have only zeros, that is, $\vec v_1 + \vec v_2 + \vec v_3 + \vec v_4 = \vec 0$. Applying Lemma~\ref{l:m}, we see that there are two pairs of equal vectors among~$v_j$. Considering the first column, we arrive at the same statement for the~$u$'s.
\end{proof}

As all~$w_j$ are supposed to be different, the necessary condition in Lemma~\ref{l:l} means that, after a re-numbering of $u$'s and~$v$'s, we must have
\begin{equation*}
 \begin{aligned}
  \vec w_1 = \vec u_1 \otimes \vec v_1, \qquad \vec w_2 = \vec u_1 \otimes \vec v_2, \\
  \vec w_3 = \vec u_2 \otimes \vec v_1, \qquad \vec w_4 = \vec u_2 \otimes \vec v_2 .
 \end{aligned}
\end{equation*}
For $i$-spin duals, this means that they must be in the vertices of a quadrilateral in coordinates $(x_2, x_3)$. We choose again the first spin to be in the simplest form, namely, $\vec w_1 = \vec w$~\eqref{s1}, so the quadrilateral's vertices are
\begin{equation}\label{xi}
 (x_2, x_3)\quad = \quad (\underbrace{1\ldots 1}_n, \underbrace{0\ldots 0}_n),\; (\underbrace{1\ldots 1}_n, \chi _3 ),\; (\chi _2, \underbrace{0\ldots 0}_n),\; (\chi _2, \chi _3),
\end{equation}
with some $n$-binary numbers
\begin{equation}\label{xi0}
\chi _2 \ne 1\ldots 1 \quad\; \text{and} \quad\; \chi _3 \ne 0\ldots 0.
\end{equation}
Note that only $(\chi _2, \chi _3)$ produces a $t$-spin dual with unities outside both the first row (nonzero~$\alpha$) and first column (nonzero~$\beta$).

\begin{lemma}\label{l:k}
Assuming~\eqref{xi0}, there is a unity in the corresponding $t$-dual outside both the first row (nonzero~$\alpha$) and first column (nonzero~$\beta$). If, besides, $(\chi _2, \chi _3) \ne (\xi _2, \xi _3)$, where
\begin{equation}\label{xi1}
 (\xi _2, \xi _3) \stackrel{\mathrm{def}}{=} (0\underbrace{1\ldots 1}_{n-1}, 1\underbrace{0\ldots 0}_{n-1}),
\end{equation}
then there exists such unity in a point $(\alpha, \beta)$ with $\alpha + \beta < 2_{10}^n$ and thus belonging, due to Proposition~\ref{p:b}, to matrix~$A$.
\end{lemma}

\begin{proof}
Assuming \eqref{xi0}, $\chi _2$ has a~$0$ at some position~$d$, and $\chi _3$ has a~$1$ at some position~$d'$; let these positions be counted from the left. For $t$-dual corresponding to $i$-dual $(\chi _2, \chi _3)$, this means that there is a unity at the point
\begin{equation}\label{ih}
(\alpha, \beta) = (\underbrace{0\ldots 0}_{d - 1} 1 \underbrace{0\ldots 0}_{n - d},\; \underbrace{0\ldots 0}_{d' - 1} 1 \underbrace{0\ldots 0}_{n - d'}).
\end{equation}

If also $(\chi _2, \chi _3) \ne (\xi _2, \xi _3)$, then at least one of $d, d'$ is greater than~$1$, and $\alpha + \beta < 2_{10}^n$ clearly holds for the point~\eqref{ih}.
\end{proof}

\begin{theorem}\label{th:theorem}
If the four $i$-spins are in the vertices of a $2^{n-1}\times 2^{n-1}$ square, then the probability for them all to take value $0\in \mathbb F_2$ is~$\frac{1}{8}$. Otherwise, the probability is $\frac{1}{16}$, so they are independent.
\end{theorem}

\begin{proof}
Consider the sum~\eqref{sum-kw} for $k=4$. It follows from what we have already calculated that, in it, $F(\vec 0)=1$ (as before), while all the summands with one, two or three~$\vec w_j$ in the argument of~$F$ are zero.

Then, Lemma~\ref{l:l} means that if the $i$-spins are \emph{not} in the vertices of a quadrilateral, then $F(\vec w_1 + \vec w_2 + \vec w_3 + \vec w_4) = 0$ immediately. If they are, but the quadrilateral is not $2^{n-1}\times 2^{n-1}$, then the same follows from Lemma~\ref{l:k}, after moving the first vertex into the first of points~\eqref{xi}. So, \eqref{sum-kw} gives the probability~$\frac{1}{16}$ in these cases, which means (together with the results for $k\le 3$ spins) that the four spins are independent.

If, however, the $i$-spins \emph{are} in the vertices of a $2^{n-1}\times 2^{n-1}$ square, then, after moving the first vertex into the first of points~\eqref{xi}, we have all the points~\eqref{xi} with, moreover, $(\chi _2, \chi _3) = (\xi _2, \xi _3)$. The $i$-spin dual that has unities exactly at these places---denote it~$\vec a$, for this moment---is
\begin{equation*}
  \vec a = \begin{pmatrix} 1 \\ 1 \end{pmatrix} \otimes \underbrace{ \begin{pmatrix} 1 \\ 0 \end{pmatrix} \otimes \cdots \otimes \begin{pmatrix} 1 \\ 0 \end{pmatrix} }_{n-1} \;\; \otimes \;\; \begin{pmatrix} 1 \\ 1 \end{pmatrix} \otimes \underbrace{ \begin{pmatrix} 1 \\ 0 \end{pmatrix} \otimes \cdots \otimes \begin{pmatrix} 1 \\ 0 \end{pmatrix} }_{n-1} ,
\end{equation*}
and having the explicit form~\eqref{BB} of $\mathbf B = \mathbf G^{-1}$, one can easily see that
\begin{equation*}
 \vec a = \mathbf B \vec b,
\end{equation*}
where $\vec b$ is the following $t$-spin dual:
\begin{equation*}
 \vec b = \begin{pmatrix} 0 \\ 1 \end{pmatrix} \otimes \underbrace{ \begin{pmatrix} 1 \\ 0 \end{pmatrix} \otimes \cdots \otimes \begin{pmatrix} 1 \\ 0 \end{pmatrix} }_{n-1} \;\; \otimes \;\; \begin{pmatrix} 0 \\ 1 \end{pmatrix} \otimes \underbrace{ \begin{pmatrix} 1 \\ 0 \end{pmatrix} \otimes \cdots \otimes \begin{pmatrix} 1 \\ 0 \end{pmatrix} }_{n-1}.
\end{equation*}
This means that $\vec b$ has just one entry~$1\in \mathbb F_2$, and its address is $(\alpha, \beta) = (1\underbrace{0\ldots 0}_{n-1}, 1\underbrace{0\ldots 0}_{n-1})$. According to Proposition~\ref{p:b}, this address belongs to~$A^{\mathrm T}$, and according to~\eqref{QAT}, it gives $F(\vec w_1 + \vec w_2 + \vec w_3 + \vec w_4) = 1$, hence the probability is~$\frac{1}{8}$, which shows an \emph{interdependence} between these spins!
\end{proof}

\subsection[Other matrices $A$]{Other matrices $\bm A$}\label{ss:rA}
 \subsubsection{Matrices with the same correlations}\label{sss:same}

There are actually twelve $3\times 3$ matrices~$A$ over~$\mathbb F_2$ (including~\eqref{cA}) having the following properties:
\begin{itemize}\itemsep 0pt
 \item[$\mathrm{(a)}$] condition~\eqref{Delta} holds,
 \item[$\mathrm{(b)}$] $A$ has exactly one row eigenvector with eigenvalue~$1$, and it has the form $\begin{pmatrix} 1 & * & * \end{pmatrix}$ (one at the first position),
 \item[$\mathrm{(c)}$] the row eigenvector of~$A^{\mathrm T}$ with eigenvalue~$1$ (it exists and is unique because of~$\mathrm{(b)}$) has the form $\begin{pmatrix} 0 & * & * \end{pmatrix}$ (zero at the first position).
\end{itemize}

The twelve matrices are:
\begin{equation}\label{6A}
 \begin{aligned}
\begin{pmatrix}
 0 & 1 & 1 \\
 0 & 0 & 1 \\
 1 & 0 & 1 \\
\end{pmatrix}, \qquad
\begin{pmatrix}
 0 & 1 & 1 \\
 0 & 0 & 1 \\
 1 & 1 & 0 \\
\end{pmatrix}, \qquad
\begin{pmatrix}
 0 & 1 & 1 \\
 1 & 0 & 1 \\
 1 & 0 & 1 \\
\end{pmatrix}, \\
\begin{pmatrix}
 1 & 1 & 1 \\
 0 & 0 & 1 \\
 1 & 0 & 1 \\
\end{pmatrix}, \qquad
\begin{pmatrix}
 1 & 1 & 1 \\
 0 & 0 & 1 \\
 1 & 1 & 0 \\
\end{pmatrix}, \qquad
\begin{pmatrix}
 1 & 1 & 1 \\
 1 & 0 & 1 \\
 1 & 0 & 1 \\
\end{pmatrix},
 \end{aligned}
\end{equation}
and six more matrices are obtained by interchanging the 2nd and 3rd rows as well as 2nd and 3rd columns in the matrices~\eqref{6A}, that is,
\begin{equation}\label{6mA}
A \mapsto H A H, \qquad H = H^{-1} = \begin{pmatrix} 1 & 0 & 0 \\ 0 & 0 & 1 \\ 0 & 1 & 0 \end{pmatrix}.
\end{equation}

A calculation shows that each of the matrices $\mathrm G_{13}$ and~$\mathrm G_{12}$~\eqref{sfA} for each of matrices~\eqref{6A}, or obtained from them according to~\eqref{6mA}, is either $\begin{pmatrix} 1 & 1 \\ 0 & 1 \end{pmatrix}$ or $\begin{pmatrix} 1 & 1 \\ 1 & 0 \end{pmatrix}$. Then, the same methods can be applied as above in this section, with exactly the same results for the $k$-spin probabilities, $k=1,2,3,4$.

\subsubsection[A simple case with other matrices~$A$]{A simple case with other matrices~$\bm A$}\label{sss:f}

There is also a case where correlations of the kind described in the beginning of this section can be calculated using the same methods, and even much easier. Namely, assume that matrix~$A$, still obeying condition~\eqref{Delta}, has, instead of $\mathrm{(b)}$ and~$\mathrm{(c)}$ in Subsubsection~\ref{sss:same}, the following property:
\begin{itemize}\itemsep 0pt
 \item[$\mathrm{(bc')}$] each of $A$ and~$A^{\mathrm T}$ has exactly one eigenvector with eigenvalue~$1$, and it has the form $\begin{pmatrix}1 & * & * \end{pmatrix}$ for each of them.
\end{itemize}
There are 26 matrices~$A$ of this kind.

Modifying calculations of Subsection~\ref{ss:pf} for this case, we get $Q(1)=0$, as in~\eqref{QA}, but also $Q'(1)=0$, instead of~\eqref{QAT} (while $Q(0)=Q'(0)=1$ always, being the probability of the ``full sample space''). It follows then immediately that \emph{all} summands in the numerator in the rhs of~\eqref{sum-kw} always vanish, except $F(\vec 0)=1$. Hence, the probability of $k$ spins all taking value~$0\in \mathbb F_2$ is~$\frac{1}{2^k}$, hence they are all independent.

Remember that this applies to spins situated as described in the beginning of this section, and cyclic boundary conditions.

\end{document}